\documentclass[letterpaper, 10 pt, conference]{ieeeconf}  

\IEEEoverridecommandlockouts                              

\usepackage{mathrsfs}
\usepackage{xcolor}

\usepackage{amsmath,amssymb,amsfonts}
\usepackage{graphicx}
\usepackage{algorithm,algorithmic}

\usepackage{hyperref}
\usepackage{graphicx}

\usepackage{textcomp}

\usepackage{setspace}
\usepackage{lipsum}

\usepackage{mathtools}
\usepackage{enumerate}

\usepackage{caption}
\usepackage{subcaption}
\usepackage{algorithm}
\usepackage{float}

\newtheorem{theorem}{Theorem}

\newtheorem{lemma}{Lemma}

\usepackage{hyperref}

\makeatletter
\let\NAT@parse\undefined
\makeatother
\usepackage{cite}

\newtheorem{innerexample}{Example}

{\begin{innerexample}[#1]\pushQED{\qed}}%
{\popQED\end{innerexample}}
\title{\LARGE \bf
Half-Global Deadbeat Parking for Dubins Vehicle
}

\author{Miroslav Krstić, Kwang Hak Kim, and Velimir Todorovski
\thanks{This work was supported in part by the Office of Naval Research under Grant No. N00014-23-1-2376, in part by the Air Force Office of Scientific Research under Grant No. FA9550-23-1-0535, and in
part by the National Science Foundation under Grant No. ECCS-2151525. The results and opinions in this paper are solely of the authors and do not reflect the position or the policy of the U.S. Government or the National Science
Foundation.}
\thanks{M. Krstić, K. Kim, and V. Todorovski are with the Department of Mechanical and Aerospace Engineering, UC San Diego, 9500 Gilman Drive, La Jolla, CA, 92093-0411, {\tt\small \{krstic,kwk001,vtodorovski\}@ucsd.edu}}%
}

\begin{document}

\maketitle
\thispagestyle{empty}
\pagestyle{empty}


\begin{abstract}
This paper presents a framework for stabilizing the Dubins vehicle model to zero in finite time (deadbeat parking) by interpreting distance as a time-like variable. We develop control laws that bring the system to a desired position and orientation even when the forward velocity cannot be directly actuated. While the controllers employ inverse-distance gains, we show that the control input remains bounded for all time. In addition to basic deadbeat parking, we incorporate safety considerations by proposing algorithms that prevent the vehicle from crossing in front of the target, enforce deceleration as it approaches the target, and guarantee parking without curb violations. The resulting methods are well-suited for missile guidance and fixed-wing pursuit, but are broadly applicable to physical systems that are represented by the Dubins vehicle model.
\end{abstract}


\section{Introduction}

A benchmark model for motion planning and control for vehicles with constant forward velocity and curvature constraints is classically the Dubins vehicle model \cite{dubins1957curves} which establishes that shortest paths for a forward-only vehicle consist of circular arcs and straight segments, a result that has since become foundational across robotics, guidance, and other nonholonomic systems. Extensions have been developed in \cite{reeds1990optimal} to cars capable of both forward and backward motion, further enriching the geometric framework for systems with nonholonomic constraints in the context of path planning. 

In addition, the stabilization of unicycle-type kinematics has been extensively studied. However, it is a notoriously difficult problem due to Brockett’s necessary condition \cite{brockett1983asymptotic}, which establishes that no smooth, time-invariant feedback law asymptotically stabilizes such systems to a point. Many have addressed this with a variety of methods including time-varying feedback \cite{pomet1992explicit,coron1993smooth}, discontinuous feedback \cite{de2000stabilization,bloch1996stabilization_slidingmode}, and hybrid control strategies \cite{hespanha1999_hybrid_stabilization,prieur2003robust}. Another widely adopted technique is the polar coordinate transformation, which bypasses Brockett’s condition by introducing a singularity into the system dynamics and has motivated a broad range of stabilizing control designs \cite{todorovski2025_CLF,Kim2025_IOC,badreddin1993fuzzy,aicardi1995,astolfi1999exponential}. In this paper, we use the polar coordinate transformation to design controllers that bring the Dubins vehicle model to a desired position and orientation in finite time (deadbeat parking) under steering-only actuation, a setting that is particularly relevant to missile guidance, fixed-wing aircraft pursuit, and other constrained vehicle applications.

In missile guidance and pursuit problems, two broad classes of control laws have emerged. The first seeks to guarantee zero terminal miss distance without regard for the interceptor’s orientation at impact. Proportional navigation (PN) and its variants, the classically predominant method, drive the line-of-sight error to zero to achieve reliable intercepts under multitudes of conditions \cite{zarchan2012tactical,siouris2004missile}. However, the classical PN algorithms ignore the terminal orientation, which may be critical in certain applications where it is strategically valuable to approach from different directions, such as a rear-end intercept to match velocities for capture, a head-on approach to maximize impact, or a broadside intercept at roughly 90 degrees to exploit vulnerabilities like blind spots from sensors or exposed flanks.

More recently, missile guidance systems have widely adopted the use of linear quadratic (LQ) optimal controllers for their capabilities of handling more complex scenarios and with optimality guarantees \cite{palumbo2010modern}. Particularly, works such as \cite{ryoo2005optimal,ryoo2006time,shaferman2008linear} have utilized the LQ formulation to specify the terminal orientation of the interceptor. The use of LQ optimal controllers, however, is limited by the requirement that the system be linear or linearized under restrictive assumptions. For instance, \cite{ryoo2005optimal} assumes a small initial line-of-sight angle, a condition that is often not satisfied in practice.

By interpreting distance-to-target as a time-like variable, we design half-global deadbeat parking controllers for the fully nonlinear Dubins vehicle model. The resulting control laws employ inverse-distance gains but remain bounded for all time. Beyond basic deadbeat parking, we extend the results to safety-critical scenarios while still guaranteeing finite-time convergence. Motivated by practical constraints, we extend the framework to guarantee: (i) parking without passing in front of the target, (ii) deceleration as the vehicle approaches the target, and (iii) parking without violating environmental boundaries such as curbs.

\section{Exploiting Distance as Time}

Consider the polar coordinate representations of the unicycle model given as
\begin{subequations}
\label{eq:unicycle_polar}
\begin{eqnarray}
\label{eq:unicycle_polar-rho}
\dot{\rho} &=& 
-v \cos\gamma\\
\label{eq:unicycle_polar-delta}
\dot{\delta} &=&  \frac{v}{\rho} \sin\gamma
\\
\label{eq:unicycle_polar-gamma}
\dot{\gamma} &=&  \frac{v}{\rho} \sin\gamma -\omega  \,,
\end{eqnarray} 
\end{subequations}
where $\rho > 0$ is the distance to the origin, $\delta,\gamma \in \mathbb{R}$ are the polar and line-of-sight angles, respectively, and $v$ and $\omega$ represent the forward velocity and steering control input, respectively. Here, if $v=\mbox{\rm const}>0$, we have the Dubins vehicle model. The transformations of the angles $\delta$ and $\gamma$ from Cartesian coordinates are defined as $\delta = \mbox{mod}(\text{{\rm atan2}}(y ,x ),2\pi) - \pi$ and $\gamma = \mbox{mod}(\text{{\rm atan2}}(y ,x )-\theta, 2\pi) - \pi$. Finally, without loss of generality, we assume the target position and orientation is the origin for the rest of this paper.

A key observation relied on by all the ideas introduced in the paper is that, in unicycles with a constant velocity $v$, as long as the LoS angle $\gamma(t)$ is made to converge, by control design, to $0$, the progression of $\rho(t)$ will be proportional to the progression of time: $\rho(t) = \rho_0-vt$, where $v\neq 0$. 

This means that our design considerations can be translated from time $t$ to the distance $\rho$. In deadbeat parking, we treat the inverse of the distance $1/ \rho$ as a time-like variable since our goal is to drive the distance $\rho$ to zero in finite time, which can be no sooner than $v/\rho_0$ because $\gamma(t)$ cannot be identically 0 but only made to converge to 0 by feedback.

The formal treatment of the distance $\rho$ as time-like, by dividing \eqref{eq:unicycle_polar-delta} and \eqref{eq:unicycle_polar-gamma} by \eqref{eq:unicycle_polar-rho}, results in
\begin{subequations}
\label{eq:unicycle_polardrho}
\begin{eqnarray}
\label{eq:unicycle_polar-deltadrho}
\frac{d\delta}{d\rho} &=& - \frac{1}{\rho}\tan\gamma \\
\label{eq:unicycle_polar-gammadrho}
\frac{d\gamma}{d\rho} &=&- \frac{1}{\rho}\tan\gamma +\frac{\omega}{v\cos\gamma}  \,.
\end{eqnarray} 
\end{subequations}

Next we give two technically straightforward but conceptually intriguing lemmas. In the first of the two, we outline the objectives of the Lyapunov-based control design when the goal is to drive the distance to zero. The second lemma establishes conditions under which $\rho(t)$ exhibits useful time-like behavior, decreasing to zero in finite time and at least as fast as linearly. 

\begin{lemma}
\label{lem1}
Let $a>0$ and let $V:[0,\rho_0]\to\mathbb{R}_{\ge 0}$ be a continuously differentiable function. Then, for $\rho\in(0,\rho_0]$,
\begin{itemize}
\item $\dfrac{dV}{d\rho} \geq \dfrac{a}{\rho}V$ implies $V(\rho)\leq V(\rho_0) \left(\dfrac{\rho}{\rho_0}\right)^a$. 
\item $\dfrac{dV}{d\rho} \geq \dfrac{a}{\rho^2}V$ implies $V(\rho)\leq V(\rho_0) {\rm e}^{a\left(1/{\rho_0}-{1}/{\rho}\right)}$. 
\end{itemize}
\end{lemma}

\begin{proof}
In Appendix.
\end{proof}

\begin{lemma}
\label{lem2}
Consider $\dot\rho = - v\cos\gamma$ for $v>0$, let $\alpha\in\mathcal{K}$, and denote
\begin{equation}
t_1 = \frac{\rho_0}{v}\sqrt{1+\alpha(1)}\,.
\end{equation}
For all $t\in [0,T)$  for which the solution exists, if $\cos\gamma(t)>0$ and $\tan^2\gamma(t)\leq \alpha(\rho(t)/\rho_0)$ for all $t\in[0,T)$, then $\rho(t)\leq \rho_0(1-t/t_1)$ for all $t\in[0,T)$.
\end{lemma}

\begin{proof}
In Appendix.
\end{proof}

\section{Deadbeat Parking Backstepping Design}

By interpreting distance as a time-like variable, we design deadbeat parking controllers for \eqref{eq:unicycle_polar} under a constant forward velocity $v$, typically referred to as the Dubins vehicle model.

\begin{theorem}
\label{thm:Dubins-FT-stabilize}
For $v=\mbox{\rm const}>0$,  consider \eqref{eq:unicycle_polar}. 
Let
\begin{eqnarray}\label{eq-control-omega}
\omega &=& \dfrac{v}{\rho}\bigl\{
\sin\gamma +\cos^2\gamma\bigl[\cos\gamma
(1+c_1c_2)\delta\nonumber \\
&&+(c_1+c_2)\sin\gamma
\bigr]\bigr\}
\end{eqnarray}
and $c_1,c_2 \geq \underline c :=\min\{c_1,c_2\} >1$. For all $\rho_0>0$, $\delta_0\in\mathbb{R}$ and $\gamma_0\in(-\pi/2,\pi/2)$ the following holds:
\begin{equation}
\label{eq-rho-bound}
\rho(t)\leq \rho_0(1-t/t_1)
\end{equation}
\begin{equation}
\label{eq-deltan-bound}
B^2(t) \leq M^2(c_1)\left( 1- 
t/t_1
\right)^{2\underline c}B^2_0
\end{equation}
\begin{align}
\label{eq-omega(t)-bound}
|\omega(t)| \leq & \frac{v}{\rho_0} \,
\sqrt{2}(1+\max\{c_1c_2,c_1+c_2\})\nonumber\\
&\times M(c_1)\left( 1- t/t_1
\right)^{\underline c -1}B_0
\end{align}
where $B_0 = B(\delta_0,\gamma_0)$, $B(\delta,\gamma) \coloneqq \sqrt{\delta^2 + \tan^2\gamma}$ and
\begin{equation}
M(s) := 
1+\frac{s^2}{2}  + s\sqrt{1+\frac{s^2}{4} }\,, \quad s\geq 0
\label{eq:M}
\end{equation}
for all $t\in\left[0, \min\left\{t_1,T\right\} \right)$, where 
\begin{equation}
t_1(\rho_0,\delta_0, \gamma_0, v,c_1)= \frac{\rho_0}{v}{\sqrt{1+M^2(c_1)B^2_0}}
\end{equation}
and $T$ is the interval of existence of the system's solutions, at which $\rho(T)=0$. In particular, for every $\Omega>0$, the angular velocity $\omega(t)$ is maintained in the interval $[-\Omega,\Omega]$ by choosing the forward velocity such that 
\begin{align}\label{eq-slowdown}
v\leq& \;\Omega\dfrac{\rho_0}{\sqrt{\delta_0^2 +\tan^2\gamma_0}}\nonumber\\
&\times\frac{1}{ \sqrt{2}M(c_1)(1+\max\{c_1c_2,c_1+c_2\})}\,.  
\end{align}
\end{theorem}

\begin{proof}
First, with the backstepping transformation $\zeta = \tan\gamma + c_1 \delta$ and defining $\omega = \frac{v\cos\gamma}{\rho} \left(\tan\gamma +\cos^2\gamma \, \bar\omega\right)$, we rewrite \eqref{eq:unicycle_polar-delta} and \eqref{eq:unicycle_polar-gamma} as
\begin{subequations}
\label{eq:unicycle_polar-zeta}
\begin{align}
\label{eq:unicycle_polar-delta-zeta}
\dot{\delta} &=  \frac{v\cos\gamma}{\rho} \left(-c_1\delta+\zeta\right) 
\\
\label{eq:unicycle_polar-gamma-zeta}
\dot\zeta &= 
\frac{v\cos\gamma}{\rho}\left(c_1\tan\gamma -\bar\omega\right) \,.
\end{align} 
\end{subequations}
Introduce the Lyapunov function $V = \delta^2 + \zeta^2$. Its derivative along the solutions of \eqref{eq:unicycle_polar-zeta} is $\dot V = 2 \frac{v\cos\gamma}{\rho}\left[
-c_1\delta^2 +\zeta\left(\delta+c_1\tan\gamma - \bar\omega\right)
\right]$.
Take the control as $\bar\omega = 
c_2 \zeta + \delta + c_1\tan \gamma $, which gives \eqref{eq-control-omega} and yields
\begin{equation}
\label{eq-Vdot-closed-loop}
\dot V = -2 \frac{v\cos\gamma}{\rho}\left(
c_1\delta^2 +c_2\zeta^2\right)\,.
\end{equation}
Dividing \eqref{eq-Vdot-closed-loop} by \eqref{eq:unicycle_polar-rho} one obtains
\begin{equation}
\frac{dV}{d\rho} = \frac{2}{\rho}\left(
c_1\delta^2 +c_2\zeta^2\right) \geq 2\underline c \frac{V}{\rho}\,. \label{eq:dVdrho}
\end{equation}
From Lemma \ref{lem1}, $V(\rho) \leq \left(\dfrac{\rho}{\rho_0}\right)^{2\underline c} V(\rho_0)$, namely,
\begin{equation}
\label{eq-tandelrho0}
\delta^2+\zeta^2 \leq \left(\frac{\rho}{\rho_0}\right)^{2\underline c} \left(\delta^2_0+\zeta^2_0\right)\,.
\end{equation}
Observing that $\delta^2 + \zeta^2 = \begin{bmatrix}\delta &\tan \gamma\end{bmatrix} P(c_1)  \begin{bmatrix}\delta &\tan \gamma\end{bmatrix}^\top$ where
\begin{equation}
P(c_1) = \begin{bmatrix}
1+c_1^2 &c_1\\c_1 &1
\end{bmatrix},
\end{equation}
and with \eqref{eq-tandelrho0}, one gets  
\begin{equation}
\label{eq-tandelrho}
\delta^2+\tan^2\gamma \leq \left(\frac{\rho}{\rho_0}\right)^{2\underline c} M^2(c_1)\left(\delta^2_0+\tan^2\gamma_0\right)\,,
\end{equation}
where $M^2(c_1) = \lambda_{\rm max}(P(c_1))/\lambda_{\rm min}(P(c_1))$ and is defined in \eqref{eq:M}. Inequality \eqref{eq-tandelrho} ensures that 
\begin{equation}
\tan^2\gamma(t) \leq M^2(c_1)\left(\frac{\rho}{\rho_0}\right)^{2\underline c} ( \delta^2_0+\tan^2\gamma_0)\,,
\end{equation}
which in turn guarantees that $\gamma(t)$ remains in $(-\pi/2,\pi/2)$ and, therefore, $\cos\gamma(t)$ remains positive. 
Using Lemma \ref{lem2}, one gets \eqref{eq-deltan-bound}. From \eqref{eq-control-omega}, 
\begin{equation}
|\omega\rho| \leq v\sqrt{2}(1+\max\{c_1c_2,c_1+c_2\})\sqrt{\delta^2+\tan^2\gamma}\,.
\end{equation}
Substituting \eqref{eq-tandelrho} and with \eqref{eq-rho-bound}, one gets 
\eqref{eq-omega(t)-bound}. Finally, from \eqref{eq-omega(t)-bound}, 
\eqref{eq-slowdown} is immediate. 
\end{proof}

Due to the division by $\rho$, the feedback \eqref{eq-control-omega} is clearly not continuously differentiable at $\rho=0$, nor even bounded if $\delta\neq 0$ or $\gamma\neq 0$. However the control $\omega(t)$ not only remains bounded but converges to zero, as stated in \eqref{eq-omega(t)-bound}, because $\delta(t), \gamma(t)\rightarrow 0$ at a faster rate than $\rho(t)\rightarrow 0$. Dividing by $\rho(t)$ is not robust to measurement noise, however, as is standard in missile guidance, the distance $\rho(t)$ fed into the feedback law \eqref{eq-control-omega} would be clipped to the lower bound $\underline\rho>0$. 

Fig.~\ref{fig:sim_thrm1} shows the simulation of the control law~\eqref{eq-control-omega} applied to system~\eqref{eq:unicycle_polar}. It is important to note that since the forward velocity is fixed at a nonzero value, the controllers must be switched off as soon as $\rho(T) = 0$. To reduce numerical errors near the boundary of the interval of existence, we instead apply a cutoff to both control inputs at $\rho(t) \leq 0.01$. Fig.~\ref{fig:trajectory_thrm1} and Fig.~\ref{fig:polar_thrm1} demonstrate that the controller successfully brings the system to zero. However, as illustrated in Fig.~\ref{fig:control_thrm1}, the control inputs exhibit abrupt jumps at the cutoff points and lack smoothness, which is attributed to the $1/\rho$ gain. In the finite-time stabilization literature, it is well known that achieving smooth stabilization near the target requires a gain of at least order $1/\rho^2$. This observation motivates the development of the control law in the next Theorem, which employs a higher-order gain in $\rho$.

\begin{figure}[t]
\centering
\begin{subfigure}[b]{\linewidth}
\centering
\includegraphics[width=\linewidth]{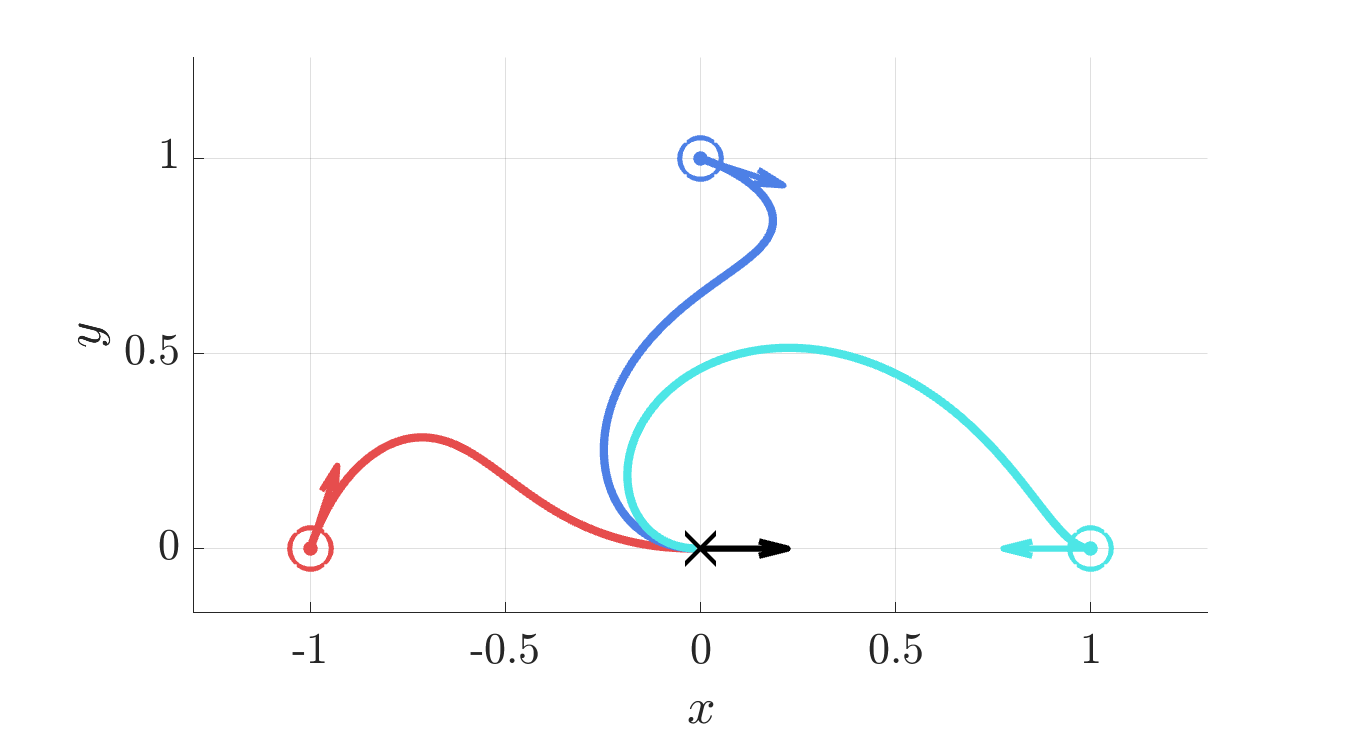}
\caption{System trajectories for the initial conditions $(\rho(0), \delta(0), \gamma(0)) = (1, 0,-\pi/2.5)$ (red), $(1,-\pi/2,-\pi/2.5)$ (blue) and $(1, \pi, 0)$ (cyan).}
\vspace{0.5em}
\label{fig:trajectory_thrm1}
\end{subfigure}
\begin{subfigure}[b]{\linewidth}
\centering
\includegraphics[width=0.8\linewidth]{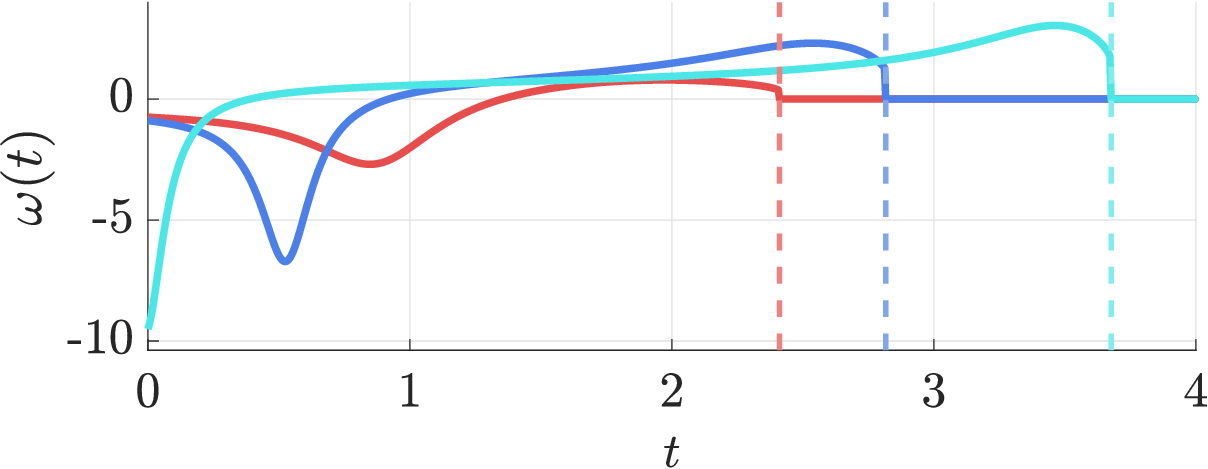}
\caption{Steering input $\omega(t)$ with cutoff ($v(t) = \omega(t) = 0$) applied when $\rho(t) \leq 0.01$. The time at which the cutoff condition is first met is indicated by the dashed vertical line.}
\vspace{0.5em}
\label{fig:control_thrm1}
\end{subfigure}
\begin{subfigure}[b]{\linewidth}
\centering
\includegraphics[width=.9\linewidth]{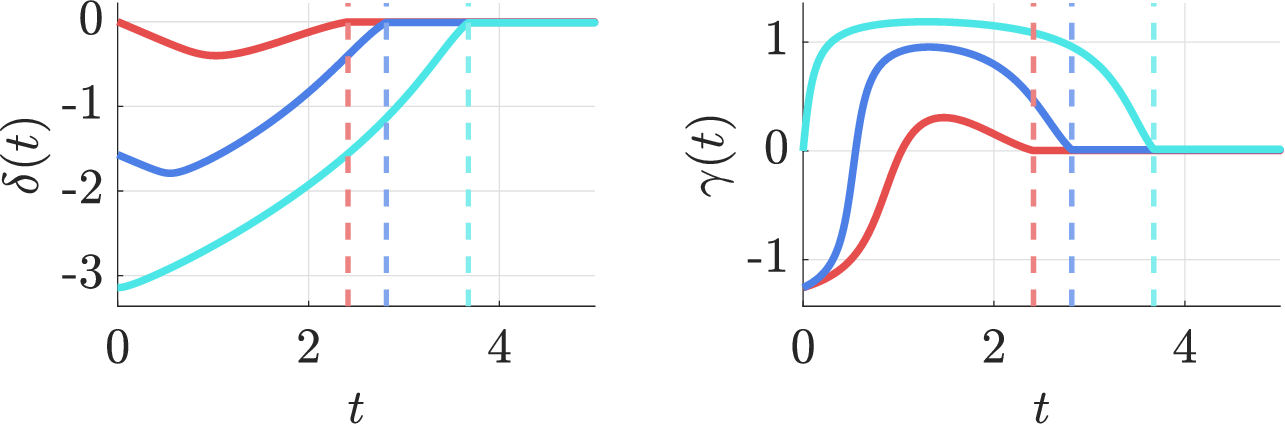}
\caption{Plots of the polar angle $\delta$ and the line-of-sight angle $\gamma$ versus time indicating the angular states reach zero before the cutoff.}
\label{fig:polar_thrm1}
\end{subfigure}
\caption{Simulation with the steering control law~\eqref{eq2-control-omega} with $c_1 = 1.01$ and $c_2 = 5$ and $v = 0.5$.}
\label{fig:sim_thrm1}
\end{figure}

\begin{theorem}
\label{thm:Dubins-FT-stabilize2}
For $v=\mbox{\rm const}>0$, consider \eqref{eq:unicycle_polar} with the controller
\begin{subequations}
\label{eq2-control-omega}
\begin{eqnarray}
\label{eq-third-controller1}
\omega &=& \frac{v}{\rho} \left(\sin\gamma +\cos^3\gamma \, \bar\omega\right)
\\
\label{eq2-control-omega2}
\bar\omega &=& \delta +\frac{c_1(\tan\gamma+\delta) +c_2\zeta}{\rho}
\\
\label{eq2-control-omega3}
\zeta &=& \tan\gamma+\frac{c_1}{\rho}\delta\,
\end{eqnarray}
\end{subequations}
and $c_1,c_2 \geq \underline c :=\min\{c_1,c_2\} > 0$. There exist 
constants $\beta_1\geq \beta_2 >0$ and $N_2(\rho_0)\geq N_1(\rho_0) >0$, which also depend on $c_1,c_2$, 
such that for all $\rho_0>0$, $\delta_0\in\mathbb{R}$ and $\gamma_0\in(-\pi/2,\pi/2)$ the following holds:
\begin{equation}
\label{eq2-rho-bound}
\rho(t)\leq \rho_0 (1-t/t_1)
\end{equation}
\begin{equation}
\label{eq2-deltan-bound}
B^2(t) \leq 
N_1 {\rm e}^{-\beta_1/({1-t/t_1})}
B^2_0
\end{equation}
\begin{equation}
\label{eq2-omega(t)-bound}
|\omega(t)| \leq v \, 
N_2 {\rm e}^{-\beta_2/({1-t/t_1})}B_0
\end{equation}
for all $t \in [0, \min\{t_1,T\})$, where
\begin{equation}
\label{eq-t1N1beta1}
t_1(\rho_0,\delta_0, \gamma_0, v,c_1,c_2) =\frac{\rho_0}{v}{\sqrt{1+N_1 {\rm e}^{-\beta_1}B^2_0}}\,
\end{equation}
and $T$ is the interval of existence of the system's solutions, at which $\rho(T)=0$, and $B_0 = B(\delta_0,\gamma_0)$, $B(\delta,\gamma) \coloneqq \sqrt{\delta^2 + \tan^2\gamma}$.
\end{theorem}

All remaining proofs are omitted for brevity and will appear in an extended journal version.

The simulation of the control law~\eqref{eq2-control-omega} applied to~\eqref{eq:unicycle_polar} is presented in Fig.~\ref{fig:sim_thrm2}. Compared to Fig.~\ref{fig:sim_thrm1}, a key difference is evident in the control input shown in Fig.~\ref{fig:control_thrm2}: the control law~\eqref{eq2-control-omega} naturally converges to zero before reaching the cutoff point, allowing the system to approach the origin smoothly. As seen in Fig.~\ref{fig:polar_thrm2}, both $\delta$ and $\gamma$ reach zero before the cutoff, eliminating the need for abrupt steering adjustments near the origin.

\begin{figure}[t]
\centering
\begin{subfigure}{\linewidth}
\centering
\includegraphics[width=0.8\linewidth]{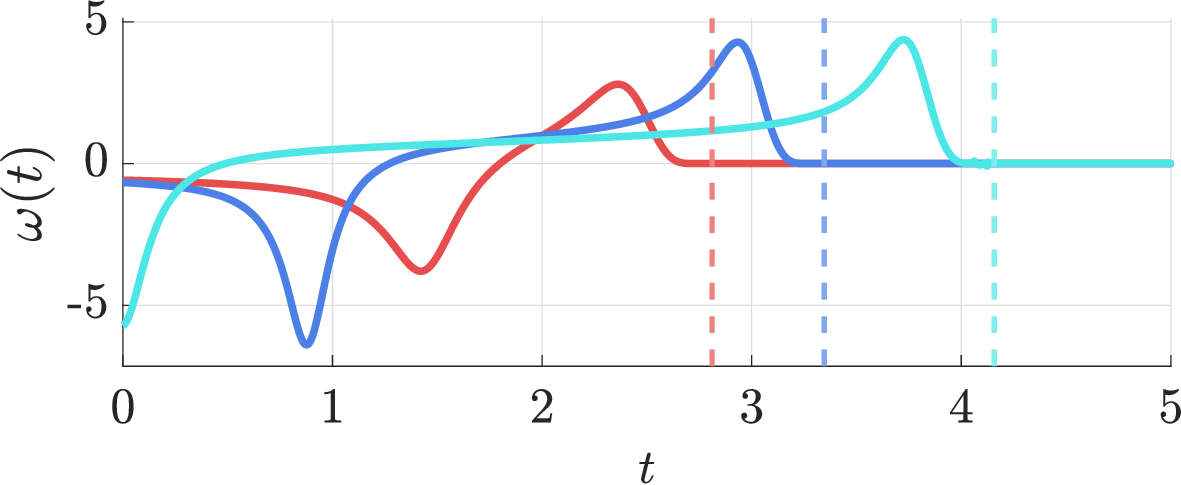}
\caption{Steering input $\omega(t)$ with cutoff ($v(t) = \omega(t) = 0$) applied when $\rho(t) \leq 0.025$. The time at which the cutoff condition is first met is indicated by the dashed vertical line.}
\vspace{0.5em}
\label{fig:control_thrm2}
\end{subfigure}
\begin{subfigure}{\linewidth}
\centering
\includegraphics[width=.9\linewidth]{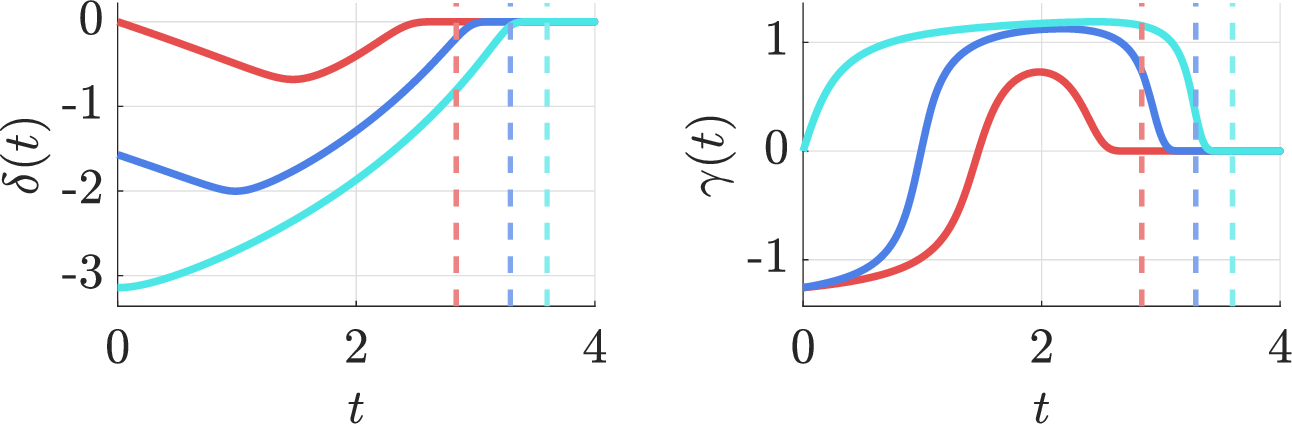}
\caption{Plots of the polar angle $\delta(t)$ and the line-of-sight angle $\gamma(t)$ indicating the angular states reach zero before the cutoff.}
\label{fig:polar_thrm2}
\end{subfigure}
\caption{Simulation with the steering control law~\eqref{eq2-control-omega} with $c_1 = c_2 = 1.2$ and $v = 0.5$. The trajectory in the $xy$-plane is nearly identical to that shown in Fig.~\ref{fig:trajectory_thrm1}, differing only in minor details, and is therefore omitted.}
\label{fig:sim_thrm2}
\end{figure}

\section{Extensions with Safety Considerations}

Beyond basic deadbeat parking, safety often plays a critical role in determining the success of the control objective. In this section, we address three safety considerations: (i) parking without crossing in front of the target, (ii) reducing speed as the system approaches the target, and (iii) parking while avoiding curb violations.

\subsection{Parking without crossing in front of the target}

Vehicle maneuvers must account not only for reaching a destination but also for how the approach is executed. For instance, in combat scenarios, passing directly in front of a target exposes the vehicle to detection, interception, or physical collision risks. Motivated by this, Theorem~\ref{thm:Dubins-FT-stabilize3} introduces a control algorithm that guarantees parking without crossing in front of the target. Its effectiveness is illustrated in Fig.~\ref{fig:thrm2vs3}, where the proposed law successfully prevents crossing the target line, in comparison to the control law~\eqref{eq2-control-omega}.

\begin{theorem}
\label{thm:Dubins-FT-stabilize3}
For $v=\mbox{\rm const}>0$, consider \eqref{eq:unicycle_polar}
with the control law \eqref{eq-third-controller1} and 
\begin{subequations}
\label{eq-third-controller}
\begin{eqnarray}
\label{eq-third-controller2}
\bar\omega &=& \left(1+\tan^2\frac{\delta}{2}\right)2\tan\frac{\delta}{2} \nonumber\\
&& +\frac{c_1(\cos\delta\tan\gamma+\sin\delta) +c_2\zeta}{\rho}
\\
\label{eq-third-controller3}
\zeta &=& \tan\gamma+\frac{c_1}{\rho}\sin\delta\,
\end{eqnarray}
\end{subequations}
with $c_1,c_2 \geq \underline c :=\min\{c_1,c_2\} > 0$. There exist 
constants $\beta_1\geq \beta_2 >0$ and $N_2(\rho_0)\geq N_1(\rho_0) >0$, which also depend on $c_1,c_2$, 
such that for all $\rho_0>0$, $\delta_0\in(-\pi,\pi)$ and $\gamma_0\in(-\pi/2,\pi/2)$ the following holds:
\begin{equation}
\label{eq2-rho-bound3}
\rho(t)\leq \rho_0(1-t/t_1)
\end{equation}
\begin{align}
\label{eq2-deltan-bound3}
B^2(t) \leq&\; 
N_1 {\rm e}^{-\beta_1/({1-t/t_1})}B^2_0
\end{align}
\begin{align}
\label{eq2-omega(t)-bound3}
|\omega(t)| \leq v 
N_2 {\rm e}^{-\beta_2/({1-t/t_1})}
(1+B_0^2) B_0 
\end{align}
for all $t\in\left[0, \min\left\{t_1,T\right\}\right)$, where 
\begin{equation}
t_1(\rho_0,\delta_0, \gamma_0, v,c_1,c_2) =\frac{\rho_0}{v}{\sqrt{1+N_1 {\rm e}^{-\beta_1}B_0^2}}\,
\end{equation}
$t_1$ is defined in \eqref{eq-t1N1beta1}, 
$T$ is the interval of existence of the system's solutions, at which $\rho(T)=0$, and $B_0=B(\delta_0,\gamma_0), B(\delta,\gamma) := \sqrt{4\tan^2\dfrac{\delta}{2} +\tan^2\gamma}$. 
\end{theorem}

\begin{figure}[t]
\centering
\includegraphics[width=0.7\linewidth]{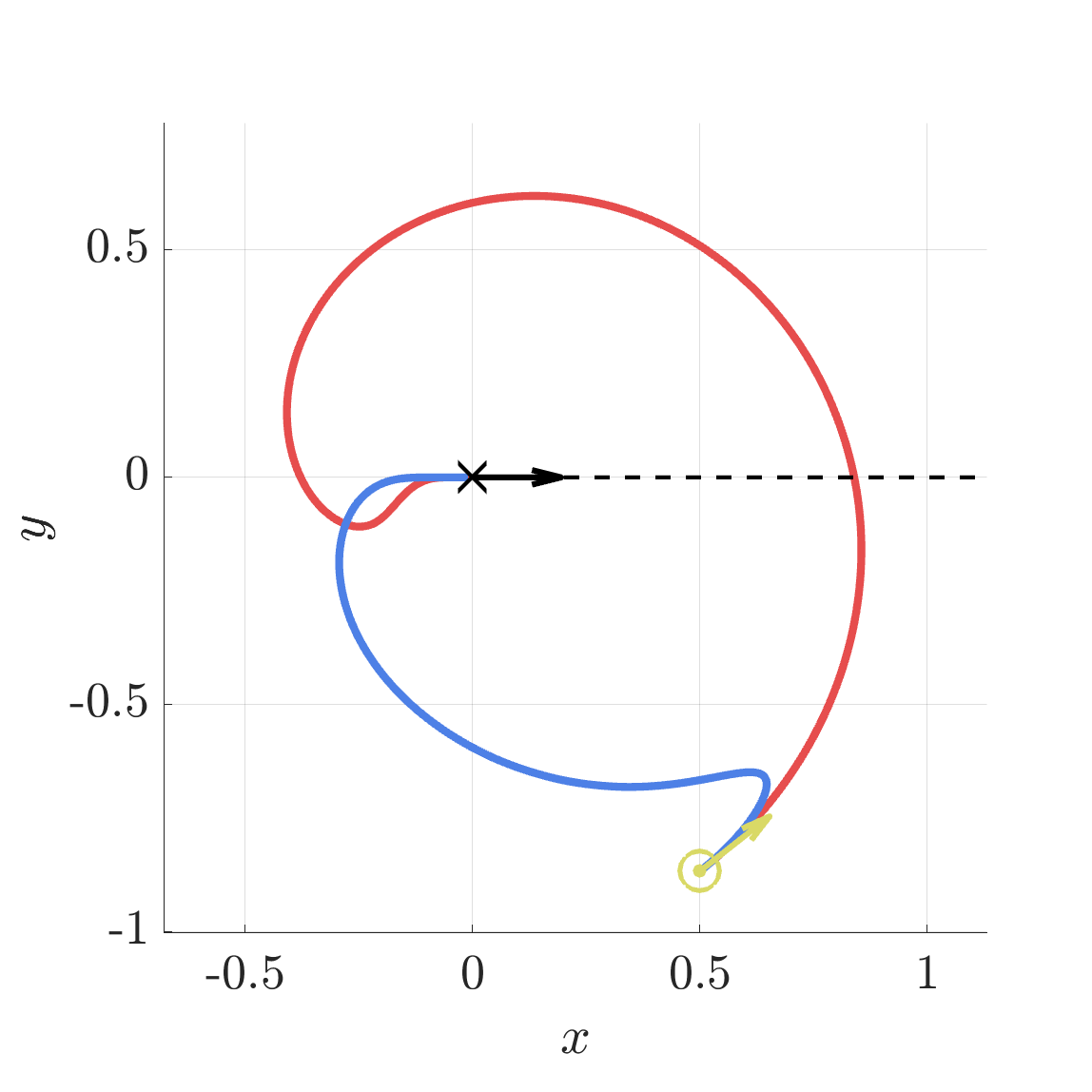}
\caption{Trajectory comparison between the steering control~\eqref{eq2-control-omega} with $v = 0.2$, $c_1 = c_2 = 0.5$ (red) and~\eqref{eq-third-controller} with $v = 0.2$, $c_1 = 1$, $c_2 = 15$ (blue), where the latter successfully avoided crossing in front of the target.}
\label{fig:thrm2vs3}
\end{figure}

\subsection{Decelerating towards the target}

If the vehicle's speed is actuated (i.e., if it is not a guided missile, but a car) then it makes sense to seek control inputs whose forward velocity decreases to zero as it approaches the target. We modify the forward velocity to that of a smooth deceleration towards the target in Theorem~\ref{thm:Dubins-FT-stabilizeslow1} with the accompanying numerical simulation in Fig.~\ref{fig:sim_thrm9} confirming the gradual decay to zero of the forward velocity input.

\begin{theorem}
\label{thm:Dubins-FT-stabilizeslow1}
For the system \eqref{eq:unicycle_polar} with the control algorithm 
\begin{subequations}
\begin{align}
v =& 
{c_0}{\rho^{n/(n+1)}}\\
\omega =& \dfrac{v}{\rho}\bigl\{
\sin\gamma +\cos^2\gamma\bigl[\cos\gamma
(1+c_1c_2)\delta\nonumber\\
&+(c_1+c_2)\sin\gamma
\bigr]\bigr\}\,,
\end{align}
\end{subequations}
where $c_0 > 0$, $c_1,c_2 \geq \underline c  >1/(n+1)$, and $ n\in\mathbb{N}$, there exists $a>0$ such that, 
for all $\rho_0>0$, $\delta_0\in\mathbb{R}$, and $\gamma_0\in(-\pi/2,\pi/2)$, 
\begin{equation}
\label{eq-rho-boundslow1}
\rho(t)\leq \rho_0(1-t/t_1)^{n+1}
\end{equation}
\begin{equation}
B(t) \leq M(c_1)B_0\left( 1- 
t/t_1\right)^{(n+1)\underline c}
\end{equation}
\begin{equation}
\label{eq-omega(t)-boundslow1}
v(t) \leq c_0 \rho_0^{n/(n+1)}(1-t/t_1)^{n}
\end{equation}
\begin{equation}
|\omega(t)| \leq  \frac{c_0a M(c_1) B_0}{\rho_0^{1/(n+1)}} 
\left( 1- t/t_1\right)^{(n+1)\underline c -1}
\end{equation}
where  $B_0 = B(\delta_0,\gamma_0)$, $B(\delta,\gamma) \coloneqq \sqrt{\delta^2+\tan^2\gamma}$, $a=\sqrt{2}(1+\max\{c_1c_2,c_1+c_2\})
M(c_1)$ and 
\begin{equation}
t_1= (n+1)\frac{\rho_0^{1/(n+1)}}{c_0}{\sqrt{1+M^2(c_1)B_0^2}}\,.
\end{equation}
\end{theorem}

\begin{figure}[t]
\centering
\includegraphics[width=.95\linewidth]{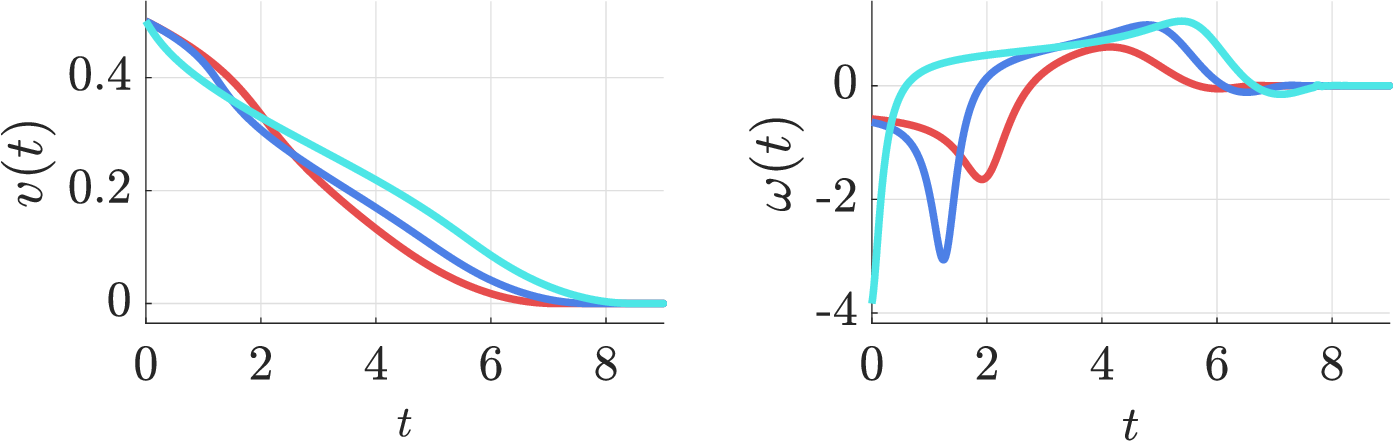}
\caption{Simulation with the control law derived in Theorem~\ref{thm:Dubins-FT-stabilizeslow1} with $c_0 = 0.5$, $c_1 = c_2 = 1.2$, and $n=2$, which makes for a smooth velocity $v(t)$. The difference is seen in $v(t)$ where the forward velocity gradually decayed to zero when approaching the target. The trajectory in the $xy$-plane is nearly identical to that shown in Fig.~\ref{fig:trajectory_thrm1}, differing only in minor details, and is therefore omitted.}
\label{fig:sim_thrm9}
\end{figure}

\subsection{Parking without curb violation}

Vehicle maneuvers are also often constrained by environmental boundaries such as curbs, barriers, or designated parking zones. To address this, in Theorem~\ref{thm:nonov2}, we incorporate boundary-aware safety conditions through the nonovershooting control framework \cite{krstic_nonovershooting_2006} that the vehicle remains strictly on one side of the half-plane with respect to the target. The simulation results in Fig.~\ref{fig:sim_nonov2} show that $\delta(t)$ remains nonnegative for all time, corresponding to $y(t) \leq 0$ for all time.

\begin{theorem}\label{thm:nonov2}
For $v=\mbox{\rm const}>0$, consider \eqref{eq:unicycle_polar}
with the control law  \eqref{eq-third-controller1} and
\begin{subequations}
\label{eq-nonovershoot2}
\begin{eqnarray}
\label{eq-nonovershoot_omegabar2}
\bar\omega &=& \frac{1}{\rho}\biggl[c_1(\sin\delta + \cos\delta\tan\gamma)\nonumber\\
&&+ c_2(1+\rho^2)\left(1+\tan^2\frac{\delta}{2}\right)^2 \zeta\biggr]
\\
\label{eq-nonovershoot_zeta2}
\zeta &=& \tan\gamma + \frac{c_1}{\rho}\sin\delta\,,
\end{eqnarray}
\end{subequations}
with $c_2 > 0$ and
\begin{align}
c_1 &> \max\left\{0,-\frac{\rho_0\tan\gamma_0}{\sin\delta_0}\right\}\label{eq:nonov2_c1}\,.
\end{align}
There exist constants $\beta_1 \geq \beta_2 > 0$ and $N_2(\rho_0) \geq N_1(\rho_0) > 0$, which also depend on $c_1,c_2$, such that for all $\rho_0 >0$, $\delta_0 \in (0,\pi)$ and $\gamma_0 \in (-\pi/2,\pi/2)$ the following holds:
\begin{align}\label{eq:nonov_rho_bound2}
\rho(t) \leq \rho_0(1-t/t_1)
\end{align}
\begin{align}\label{eq:nonov_deltan_bound2}
B^2(t)\leq N_1e^{-\beta_1/(1-t/t_1)}B^2_0
\end{align}
\begin{align}\label{eq:nonov_omega_bound2}
|\omega(t)| \leq vN_2e^{-\beta_2/(1-t/t_1)}\left(1+B_0^4\right)B_0
\end{align}
and $\delta(t) \in [0,\pi)$ for all $t \in [0, \min\{t_1,T\})$, where 
\begin{align}
t_1
= \frac{\rho_0}{v}\sqrt{1+ N_1e^{-\beta_1}B_0^2}
\end{align}
and $T$ is the interval of existence of the system's solutions, at which $\rho(T) = 0$ , and $B_0=B(\delta_0,\gamma_0), B(\delta,\gamma) := \sqrt{\tan^2\dfrac{\delta}{2} +\tan^2\gamma}$. 
\end{theorem}

\begin{figure}[t]
\centering
\begin{subfigure}{\linewidth}
\centering
\includegraphics[width=.9\linewidth]{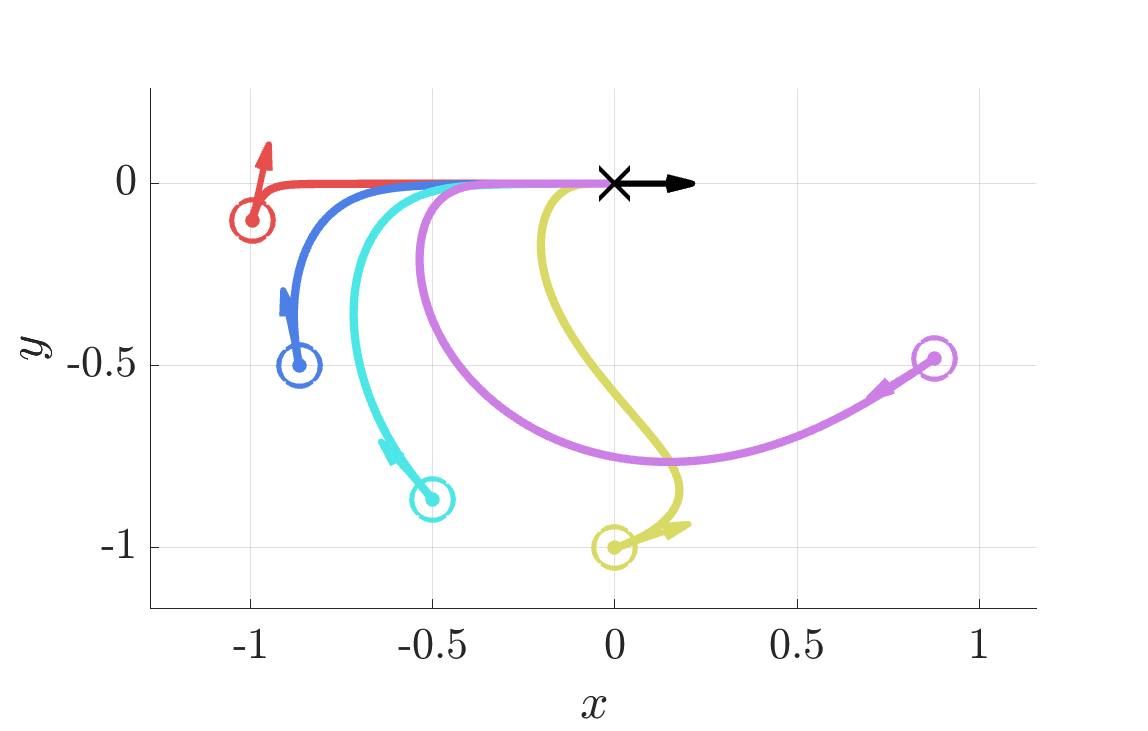}
\caption{The trajectories clearly demonstrate that the system remains in the lower half of the $xy$-plane.}
\vspace{0.5em}
\label{fig:trajectory_nonov2}
\end{subfigure}
\begin{subfigure}{\linewidth}
\centering
\includegraphics[width=0.85\linewidth]{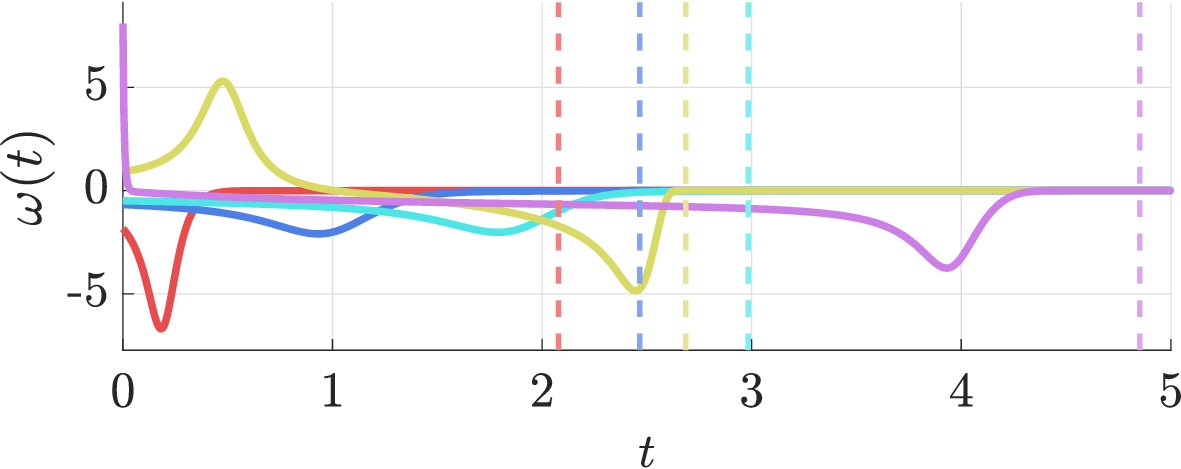}
\caption{Steering input $\omega(t)$ with cutoff ($v(t) = \omega(t) = 0$) applied when $\rho(t) \leq 0.001$. The time at which the cutoff condition is first met is indicated by the dashed vertical line.}
\vspace{0.5em}
\label{fig:control_nonov2}
\end{subfigure}
\begin{subfigure}{\linewidth}
\centering
\includegraphics[width=.9\linewidth]{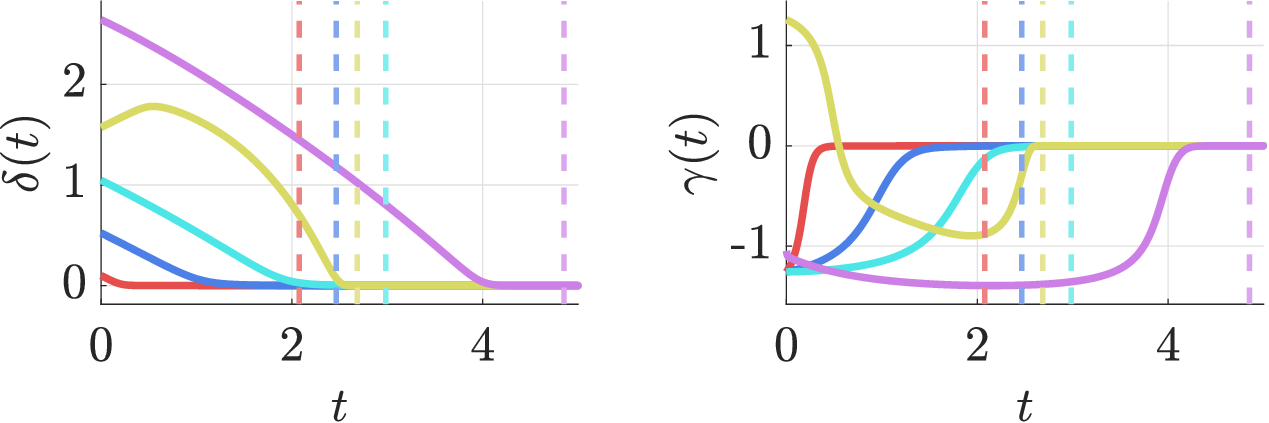}
\caption{The plots of the polar angle $\delta$ and the line-of-sight angle $\gamma$ clearly illustrate the nonundershooting behavior of $\delta(t)$, which remains nonnegative for all time. Additionally, $\gamma(t)$ does not exhibit any abrupt, near-discontinuous jumps at $\delta(t) = 0$.}
\label{fig:polar_nonov2}
\end{subfigure}
\caption{Simulation with the control law derived in Theorem~\ref{thm:nonov2} with $c_1$ chosen as~\eqref{eq:nonov2_c1}, $c_2 = 1$, and $v = 0.5$.}
\label{fig:sim_nonov2}
\end{figure}

\section{Interpretation of the Designs with Distance in the Role of Time}

It would have been confusing, or at least distracting, to initially introduce our designs through the treatment of $\ln(\rho_0/\rho)$ and $1/\rho$ as time-advancing quantities. But it is important to see this connection ex post facto. 

Consider the two models
\begin{eqnarray}
\label{eq-modellnrho0rho}
&\dfrac{d\delta}{d\ln(\rho_0/\rho)} = \tan\gamma\,, \quad \dfrac{d\tan\gamma}{d\ln(\rho_0/\rho)} = -\bar\omega&
\\
\label{eq-model1rho}
&\dfrac{d\delta}{d(1/\rho)} = \rho\tan\gamma\,, \quad \dfrac{d\tan\gamma}{d(1/\rho)} = -\rho\bar\omega\,.&
\end{eqnarray}
From the model \eqref{eq-modellnrho0rho}, 
the design in Theorem~\ref{thm:Dubins-FT-stabilize} follows, 
and from the model \eqref{eq-model1rho}, 
the design in Theorem~\ref{thm:Dubins-FT-stabilize2}. In these models, $\ln(\rho_0/\rho)$ and $1/\rho$ are time-advancing. If one wishes to constrain $\delta$ to the interval $(-\pi,\pi)$, one writes the first ODE for $2\tan(\delta/2)$, such as, for example, the model
\begin{subequations}
\label{eq-modelrhotandelta}
\begin{align}
\dfrac{d\left(2\tan\frac{\delta}{2}\right)}{d(1/\rho)} &= \rho\frac{1}{\cos^2\frac{\delta}{2}}\tan\gamma =\rho\frac{2\tan\frac{\delta}{2}}{\sin\delta}\tan\gamma\\
\dfrac{d\tan\gamma}{d(1/\rho)} &= -\rho\bar\omega\,,   
\end{align}
\end{subequations}
from which the design in Theorem \ref{thm:Dubins-FT-stabilize3} follows. 

The point of view with the distance $\rho$ in a time-like role doesn't make the designs easier, especially for the ``time-varying'' models \eqref{eq-model1rho} and \eqref{eq-modelrhotandelta}, with coefficients dependent on $\rho$. But these alternative models do enhance the understanding of the problem and design challenges. For example, due to the standard ``double integrator'' structure in the model \eqref{eq-modellnrho0rho}, the only challenge is in the closed-loop analysis, after the feedback design (by backstepping or forwarding, for the double integrator). 

The model \eqref{eq-modellnrho0rho} indicates that the Dubins vehicle \eqref{eq:unicycle_polar} is, in fact, feedback-linearizable, using steering $\omega$ as control, in an appropriate time scale. For example, in the time scale $\ln(\rho_0/\rho)$, the transformation $(\delta,\gamma)\mapsto (\delta,\tan\gamma)$ is a local ($|\gamma|<\pi/2$) diffeomorphism which, along with control \eqref{eq-third-controller1}, transforms the Dubins vehicle model \eqref{eq:unicycle_polar} into the linear system \eqref{eq-modellnrho0rho}, along with nonlinear time-varying ``zero dynamics''
\begin{equation}
\dfrac{d(t/\rho_0)}{d\ln(\rho_0/\rho)} = \frac{1}{v}{\rm e}^{-\ln(\rho_0/\rho)} \sqrt{1+\tan^2\gamma}\,,
\end{equation}
where the scaled time $t/\rho_0$ plays the role of a state variable and $\ln(\rho_0/\rho)$ plays the role of advancing time. If a controller is designed to ensure that $\sqrt{\delta^2+\tan^2\gamma}\leq M \sqrt{\delta_0^2+\tan^2\gamma_0} $ for $M\geq 1$, the ``zero dynamics state'' $t$ cannot grow above $t_1 =\frac{\rho_0}{v}\sqrt{1+M^2(\delta_0^2+\tan^2\gamma_0)}$. 


\section{Conclusion}

We leverage the interpretation of distance as a time-like variable to design deadbeat parking control laws for the Dubins vehicle model, particularly relevant when the forward velocity cannot be directly actuated. Although the control laws employ an inverse-distance gain, we establish that the resulting control input remains bounded for all time. These results extend to safety-critical scenarios in which the algorithms ensure parking without crossing in front of the target, enforce deceleration as the system approaches the target, and prevent curb violations. While the current formulation is restricted to a stationary target, in upcoming work, we consider moving targets and, ultimately, pursuer–evader games.


\bibliographystyle{IEEEtranS}
\bibliography{root}

\begin{thebibliography}{10}
\providecommand{\url}[1]{#1}
\csname url@samestyle\endcsname
\providecommand{\newblock}{\relax}
\providecommand{\bibinfo}[2]{#2}
\providecommand{\BIBentrySTDinterwordspacing}{\spaceskip=0pt\relax}
\providecommand{\BIBentryALTinterwordstretchfactor}{4}
\providecommand{\BIBentryALTinterwordspacing}{\spaceskip=\fontdimen2\font plus
\BIBentryALTinterwordstretchfactor\fontdimen3\font minus \fontdimen4\font\relax}
\providecommand{\BIBforeignlanguage}[2]{{%
\expandafter\ifx\csname l@#1\endcsname\relax
\typeout{** WARNING: IEEEtranS.bst: No hyphenation pattern has been}%
\typeout{** loaded for the language `#1'. Using the pattern for}%
\typeout{** the default language instead.}%
\else
\language=\csname l@#1\endcsname
\fi
#2}}
\providecommand{\BIBdecl}{\relax}
\BIBdecl

\bibitem{aicardi1995}
M.~Aicardi, G.~Casalino, A.~Bicchi, and A.~Balestrino, ``Closed loop steering of unicycle like vehicles via lyapunov techniques,'' \emph{IEEE Robotics \& Automation Magazine}, vol.~2, no.~1, pp. 27--35, 1995.

\bibitem{astolfi1999exponential}
A.~Astolfi, ``Exponential stabilization of a wheeled mobile robot via discontinuous control,'' \emph{Journal of Dynamic Systems, Measurement, and Control}, vol. 121, no.~1, pp. 121--126, 03 1999.

\bibitem{badreddin1993fuzzy}
E.~Badreddin and M.~Mansour, ``Fuzzy-tuned state-feedback control of a non-holonomic mobile robot,'' \emph{IFAC Proceedings Volumes}, vol.~26, no.~2, pp. 769--772, 1993.

\bibitem{bloch1996stabilization_slidingmode}
A.~Bloch and S.~Drakunov, ``Stabilization and tracking in the nonholonomic integrator via sliding modes,'' \emph{Systems \& Control Letters}, vol.~29, no.~2, pp. 91--99, 1996.

\bibitem{brockett1983asymptotic}
R.~W. Brockett \emph{et~al.}, ``Asymptotic stability and feedback stabilization,'' \emph{Differential geometric control theory}, vol.~27, no.~1, pp. 181--191, 1983.

\bibitem{coron1993smooth}
J.-M. Coron and B.~D'Andrea-Novel, ``Smooth stabilizing time-varying control laws for a class of nonlinear systems. application to mobile robots,'' in \emph{Nonlinear Control Systems Design 1992}.\hskip 1em plus 0.5em minus 0.4em\relax Elsevier, 1993, pp. 413--418.

\bibitem{de2000stabilization}
A.~De~Luca, G.~Oriolo, and M.~Vendittelli, ``Stabilization of the unicycle via dynamic feedback linearization,'' \emph{IFAC Proceedings Volumes}, vol.~33, no.~27, pp. 687--692, 2000.

\bibitem{dubins1957curves}
L.~E. Dubins, ``On curves of minimal length with a constraint on average curvature, and with prescribed initial and terminal positions and tangents,'' \emph{American Journal of mathematics}, vol.~79, no.~3, pp. 497--516, 1957.

\bibitem{hespanha1999_hybrid_stabilization}
J.~P. Hespanha and A.~S. Morse, ``Stabilization of nonholonomic integrators via logic-based switching,'' \emph{Automatica}, vol.~35, no.~3, pp. 385--393, 1999.

\bibitem{Kim2025_IOC}
K.~H. Kim, V.~Todorovski, and M.~Krstic, ``Inverse optimal feedback and gain margins for unicycle stabilization,'' \emph{to be posted on arXiv}, 2025.

\bibitem{krstic_nonovershooting_2006}
M.~Krstic and M.~Bement, ``\BIBforeignlanguage{en}{Nonovershooting {Control} of {Strict}-{Feedback} {Nonlinear} {Systems}},'' \emph{\BIBforeignlanguage{en}{IEEE Transactions on Automatic Control}}, vol.~51, no.~12, pp. 1938--1943, Dec. 2006.

\bibitem{palumbo2010modern}
N.~F. Palumbo, R.~A. Blauwkamp, and J.~M. Lloyd, ``Modern homing missile guidance theory and techniques,'' \emph{Johns Hopkins APL technical digest}, vol.~29, no.~1, pp. 42--59, 2010.

\bibitem{pomet1992explicit}
J.-B. Pomet, ``Explicit design of time-varying stabilizing control laws for a class of controllable systems without drift,'' \emph{Systems \& control letters}, vol.~18, no.~2, pp. 147--158, 1992.

\bibitem{prieur2003robust}
C.~Prieur and A.~Astolfi, ``Robust stabilization of chained systems via hybrid control,'' \emph{IEEE Transactions on Automatic Control}, vol.~48, no.~10, pp. 1768--1772, 2003.

\bibitem{reeds1990optimal}
J.~Reeds and L.~Shepp, ``Optimal paths for a car that goes both forwards and backwards,'' \emph{Pacific journal of mathematics}, vol. 145, no.~2, pp. 367--393, 1990.

\bibitem{ryoo2005optimal}
C.-K. Ryoo, H.~Cho, and M.-J. Tahk, ``Optimal guidance laws with terminal impact angle constraint,'' \emph{Journal of Guidance, Control, and Dynamics}, vol.~28, no.~4, pp. 724--732, 2005.

\bibitem{ryoo2006time}
------, ``Time-to-go weighted optimal guidance with impact angle constraints,'' \emph{IEEE Transactions on control systems technology}, vol.~14, no.~3, pp. 483--492, 2006.

\bibitem{shaferman2008linear}
V.~Shaferman and T.~Shima, ``Linear quadratic guidance laws for imposing a terminal intercept angle,'' \emph{Journal of Guidance, Control, and Dynamics}, vol.~31, no.~5, pp. 1400--1412, 2008.

\bibitem{siouris2004missile}
G.~M. Siouris, \emph{Missile guidance and control systems}.\hskip 1em plus 0.5em minus 0.4em\relax Springer, 2004.

\bibitem{todorovski2025_CLF}
V.~Todorovski, K.~H. Kim, and M.~Krstic, ``Modular design of strict control lyapunov functions for global stabilization of the unicycle in polar coordinates,'' \emph{to be posted on arXiv}, 2025.

\bibitem{zarchan2012tactical}
P.~Zarchan, \emph{Tactical and strategic missile guidance}.\hskip 1em plus 0.5em minus 0.4em\relax American Institute of Aeronautics and Astronautics, Inc., 2012.

\end{thebibliography}

\appendix
\section{Appendix A}
\label{appA}

\begin{proof}[Proof of Lemma \ref{lem1}]
For $\rho\in(0, \rho_0]$, $\dfrac{dV}{d\rho} \geq \dfrac{a}{\rho}V$ can be rewritten as $\dfrac{dV}{d\ln(\rho_0/\rho)}\leq -aV$ and $\dfrac{dV}{d\rho} \geq \dfrac{a}{\rho^2}V$ as $\dfrac{dV}{d(\rho/\rho_0-1)} \leq - \dfrac{a}{\rho_0}V$, from which the results follow by the comparison principle.
\end{proof}

\begin{proof}[Proof of Lemma \ref{lem2}]
In this proof we use the trigonometric identity $-|\cos\gamma| = - 1/\sqrt{1+\tan^2\gamma}$. For $\cos\gamma>0$ it follows that $\dot\rho\leq - v/\sqrt{1+\alpha(\rho/\rho_0)}$, which implies that $\rho/\rho_0\leq 1$ and hence $\dot\rho\leq - v/\sqrt{1+\alpha(1)}$. With the comparison principle, the upper bound $\rho(t)\leq \rho_0(1-t/t_1)$ follows. 
\end{proof}

\end{document}